\documentclass[10pt,letterpaper]{amsart}

\usepackage{amsmath,amsthm,amsfonts,amssymb}
\usepackage{stmaryrd}
\usepackage{mathtools}
\usepackage{mathrsfs}
\usepackage[mathcal]{euscript}
\usepackage[pdftex,bookmarks=true]{hyperref}
\usepackage[thinlines]{easybmat}
\usepackage{amsrefs}

\newcommand{\R}{\mathbb{R}}
\newcommand{\Z}{\mathbb{Z}}
\DeclareMathOperator{\Gl}{GL}
\DeclareMathOperator{\Sl}{SL}
\DeclareMathOperator{\UT}{UT}
\DeclareMathOperator{\Or}{O}
\DeclareMathOperator{\Nil}{Nil}
\DeclareMathOperator{\Rm}{Rm}
\DeclareMathOperator{\Ric}{Ric}
\DeclareMathOperator{\Rc}{Rc}

\DeclareMathOperator{\Lie}{Lie}
\DeclareMathOperator{\ad}{ad}

\newcommand{\mfr}[1]{\ensuremath \mathfrak{#1}}
\newcommand{\mcl}[1]{\ensuremath \mathcal{#1}}
\newcommand{\map}[2]{\nolinebreak{#1 \rightarrow #2}}
\newcommand{\mapelts}[2]{\nolinebreak{#1 \mapsto #2}}
\newcommand{\fn}[3]{\nolinebreak{#1 \colon #2 \rightarrow #3}}
\newcommand{\fnl}[3]{\nolinebreak{#1 \colon #2 \longrightarrow #3}}
\newcommand{\fndef}[5]{\begin{align*}\fnla{#1}{#2}{#3} \\ \mapeltsla{#4}{#5} \end{align*}}
\newcommand{\fnla}[3]{#1 \colon #2 &\longrightarrow #3} %for proper alignment
\newcommand{\mapeltsla}[2]{#1 &\longmapsto #2}
\newcommand{\setvbigl}[2]{\nolinebreak{ \left\{ \left. #1\hphantom{,} \right| \hphantom{,}#2 \right\}}}

\newtheorem{thm}{Theorem}[section]
\newtheorem*{thm*}{Theorem}
\newtheorem{cor}[thm]{Corollary}
\newtheorem*{cor*}{Corollary}
\newtheorem{lem}[thm]{Lemma}
\newtheorem*{lem*}{Lemma}
\newtheorem{prop}[thm]{Proposition}
\newtheorem*{prop*}{Proposition}

\theoremstyle{remark}
\newtheorem*{rem}{Remark}

\theoremstyle{definition}

\theoremstyle{definition}
\newtheorem{exmp}[thm]{Example}

\begin{document}

\title{Explicit {R}icci solitons on nilpotent {L}ie groups}
%\date{\today}

\author[M.~B.~Williams]{Michael Bradford Williams}
\address{Department of Mathematics \\ The University of Texas at Austin}
\email{mwilliams@math.utexas.edu}
\urladdr{http://www.ma.utexas.edu/users/mwilliams/}

%53C44    	Geometric evolution equations (mean curvature flow, Ricci flow, etc.)
%22E25    	Nilpotent and solvable Lie groups

\begin{abstract}
We consider Ricci flow on two classes of nilpotent Lie groups that generalize the three-dimensional Heisenberg group:~ the higher-dimensional classical Heisenberg groups, and the groups of real unitriangular matrices.  Each group is known to admit a Ricci soliton, but we construct them \textit{explicitly} on each group.  In the first case, this is done using Lott's blowdown method, whereby we demonstrate convergence of arbitrary diagonal metrics to the solitons.  In the second case, which is more complicated, we obtain the solitons using a suitable ansatz.
\end{abstract}

%%%-------------------------------------------------------------------
%%% body of the paper
%%%-------------------------------------------------------------------

\maketitle
\tableofcontents

\section{Introduction}\label{sec-intro}

The three-dimensional Heisenberg group, also known as $\Nil^3$, is one of Thurston's model geometries.  It is important in understanding the structure of manifolds and Ricci solitons in dimension three, and in understanding Ricci flow as a dynamical system for three-dimensional metric Lie algebras \cite{GlickPayne2009}.  As such, it has been studied extensively: for example, by Isenberg and Jackson \cite{IsenbergJackson1992}, Knopf and McLeod \cite{KnopfMcLeod2001}, Lott \cite{Lott2007}, Baird and Danielo \cite{Baird2007}, and Glickenstein \cite{Glick2008}.  While the behavior of Ricci flow on this Lie group is well-understood, the understanding of solitons on general nilpotent Lie groups (e.g., in higher dimensions) is nascent.

The group $\Nil^3$ is an example of a \textit{nilmanifold}, which is a nilpotent Lie group $N$ together with a left-invariant metric $g$.  We say $g$ is a \textit{nilsoliton} if $g(t) = c(t) \, \eta_t^* g$ is a solution of Ricci flow, for a function $c(t)$ and a one-parameter family of diffeomorphisms $\{\eta_t\}$ of $N$.  The nilsoliton is \textit{expanding} if $c(t) = t$.  In \cite{Lauret2001}, Lauret showed that a metric $g$ on a nilmanifold is a nilsoliton if and only if $\Ric_g = cI + D$, where $\Ric_g$ is the Ricci endomorphism of $g$, $c \in \R$, and $D$ is derivation of the Lie algebra $\mfr{n}$.  Not all nilmanifolds admit nilsolitons, but Lauret also showed (among many other things) that when they exist, they are unique up to isometry and scaling.  More recently, he has shown in \cite{Lauret2010-nil} that solutions of Ricci flow on nilmanifolds exist for all time, and are Type-III: $\|\Rm(g(t))\| \leq C/t$ for $t > 0$.  The methods used involve a flow (equivalent to Ricci flow) on the space of nilpotent Lie brackets, where the algebraic structure is more prominent.

Despite this, there are still very few explicit examples of Ricci solitons on nilpotent Lie groups (to say nothing of Ricci flow solutions in general), and that is the motivation for this paper.  As we will see, Lauret has answered the question, ``Do they exist?''  We focus on the questions, ``What are they?''~ and ``How do they behave?''  Namely, we demonstrate explicit nilsoliton metrics on two classes of nilpotent Lie groups that generalize $\Nil^3$ to higher dimensions.  For one class, we also show that arbitrary diagonal metrics will converge, modulo rescaling, to such solitons.  

\begin{thm}\label{thm:heis} Let $H_N \R$ be the classical Heisenberg group of dimension $N=2n+1$, with coordinates $(x^i)$ and coframe $\{\theta^i\}$ to be described later.  Let $g(t)$ be Ricci flow solution on $H_N \R$, starting at a diagonal left-invariant metric $g_0$.
\begin{itemize}
\item[(a)]\label{heis-a} The solution $g(t)$ has the following asymptotic behavior:
\begin{align*}
g_{ii}      &\sim \gamma_i t^{1/n+2} \\
g_{NN}      &\sim \gamma_N t^{-n/n+2}
\end{align*}
where $i=1,\dots,2n$, and the $\gamma$s are constants depending only on $g_0$ and $n$.

\item[(b)]\label{heis-b} The solution $g(t)$ converges, after pullback by diffeomorphisms, to the solution $g_{\infty}(t)$ corresponding to the metric
\[ g_{\infty} = \theta^1 \otimes \theta^1 + \cdots + \theta^{2n} \otimes \theta^{2n} + \frac{1}{n+2} \, \theta^N \otimes \theta^N. \]
This is a nilsoliton with respect to the diffeomorphisms
\[ \eta_t(x^1,\dots,x^{2n},x^N) = (t^{-\frac{1}{2} \frac{n+1}{n+2}} x^1,\dots,t^{-\frac{1}{2} \frac{n+1}{n+2}} x^{2n},t^{-\frac{n+1}{n+2}} x^N). \]
\end{itemize}
\end{thm}

\begin{thm}\label{thm:uni} Let $\UT_n \R$ be the Lie group of real $n \times n$ unitriangular matrices, with coordinates $(x^{ij})$ and coframe $\{\theta^{ij}\}$ to be described later.  Then the family of diagonal metrics $g(t) = g_{ij,ij}(t) \, \theta^{ij} \otimes \theta^{ij}$, where 
\[ g_{ij,ij}(t) = \frac{1}{n^{j-i-1}} t^{1-2(j-i)/n}, \]
is a Ricci flow solution on $\UT_n \R$.  The metric $g(1)$ is a nilsoliton with respect to the diffeomorphisms $\eta_t$, where
\[ (\eta_t(x))^{ij} = t^{-(j-i)/n} x^{ij}. \]
\end{thm}

We make a few explanatory remarks regarding these results.  The solitons on these spaces were shown to exist by Lauret in \cite{Lauret2001}.  The solitons in Theorem~\ref{thm:heis} are studied by Payne \cite{Payne2010c} in the context of a related but distinct evolution equation, the ``projectivized bracket flow,'' introduced in that paper. The existence of the metric $g(1)$ in Theorem~\ref{thm:uni} may also be deduced from results in \cite{Tamaru2007}. The (mostly analytic) approach taken in the present paper provides the following additional features, which facilitate the study of these solitons as models of infinite-time (non-homogeneous) Ricci flow solutions undergoing collapse: 
\begin{itemize}
\item[(1)] We demonstrate the existence of an explicit stably Ricci-diagonal basis\footnote{A \textit{stably Ricci-diagonal basis} is a basis of the Lie algebra (equivalently, a left-invariant frame) such that the Ricci tensors of any family of diagonal metrics are all diagonal.  Such bases do not always exist.} for both families of metrics.
\item[(2)] We construct explicit families of diffeomorphisms that exhibit both families of solitons --- \emph{a priori} solutions of a static elliptic system --- as time-dependent solutions of the Ricci flow parabolic system.
\item[(3)] We deduce the asymptotic behaviors of solutions $g(t)$ in Theorem~\ref{thm:heis}, which do not readily follow from the corresponding results for projectivized bracket flow.
\end{itemize}

The structure of this paper is as follows.  In Section \ref{sec-blowdown} we recall Lott's blowdown method for finding solitons, and review the three-dimensional example case.  Section \ref{sec-heis} focuses on Heisenberg groups.  In it, we examine more closely the structure of the classical Heisenberg groups, and compute their Ricci tensors.  Using these computations, we write down the Ricci flow equations and describe the asymptotic behavior of solutions.  Then we find the solitons with the blowdown method.  We conclude with analysis of collapse of compact quotients of Heisenberg groups, interpreted as Riemannian groupoids.  

Section \ref{sec-uni} focuses on the (significantly more complicated) spaces of unitriangular matrices.  We review essential properties of these spaces, and compute their Ricci tensors with the aid of a computer algebra system.  Finally, we analyze the Ricci flow, and construct the Ricci solitons.

The appendix contains the derivation of some helpful formulas for curvature of Lie groups.

\section{The blowdown method}\label{sec-blowdown}

In this section, we recall a method for finding solitons that Lott used extensively in \cite{Lott2007}.  We also review the Heisenberg soliton in three dimensions.  As mentioned above, this example appears in several other places, but we include it here for completeness, to establish notation, and to motivate the procedures (adapted from the above references) that we will use in the general case.

Let $M$ be a manifold with local coordinates $(x^1,\dots,x^n)$, local frame $\{F_1,\dots,F_n\}$, and dual coframe $\{\theta^1,\dots,\theta^n\}$.  Suppose that $(M,\hat{g}(t))$ is a type III Ricci flow solution such that the metric $\hat{g}(t)$ stays diagonal, and that its asymptotic behavior is given by some other metric $g(t)$.  We write 
\[ g(t) = g_i(t) \, \theta^i \otimes \theta^i, \]
where\footnote{We use the symbol $\sim$ to mean $a(t) \sim b(t)$ if and only if ${\displaystyle \lim_{t \rightarrow \infty} \frac{a(t)}{b(t)} = 1}$.} $\hat{g}_i(t) \sim g_i(t)$ for all $i=1,\dots,n$.
Consider the \textit{blowdown} of this solution,
\[ g_s(t) = \frac{1}{s} g(st), \]
which itself is another Ricci flow solution.  The behavior of $g_s(t)$ as $s \rightarrow \infty$ tells us about the behavior of the original solution $g(t)$ whenever $t$ is large.  %\marginpar{length scales}

Note that it does not matter in which order we take a blowdown or find asymptotics.  Namely,
\[ \hat{g}_i(t) \sim g_i(t) \Longleftrightarrow \frac{1}{s} \hat{g}_i(st) \sim \frac{1}{s} g_i(st). \]

The goal is to find a family of diffeomorphisms $\{\fn{\phi_s}{M}{M}\}_{s>0}$, such that $\phi_s^* g_s(t)$ is a Ricci flow solution for each $s$, and such that
\[ g_{\infty}(t) = \lim_{s \rightarrow \infty} \phi_s^* g_s(t) \]
exists.  By Proposition 2.5 in \cite{Lott2007}, this limit (whenever it exists) is a soliton metric on $M$.  

Note that for the above limit to exist, it is necessary that $\phi_s^* g_i(st)/s$ is finite and positive for each fixed $s$ and $t$.  In explicit calculations, it is extrememly helpful to choose the family $\{\phi_s\}$ such that
\[ \phi_s^* \theta^i = \alpha^i(s) \, \theta^i \]
for all $i$ and for some functions $\alpha^i(s)$.  This is usually straight-forward when the solution is diagonal.

\begin{exmp}\label{nil3}
Consider the Lie group
\[ \Nil^3 = \setvbigl{\begin{pmatrix} 1 & x & z \\ 0 & 1 & y \\ 0 & 0 & 1 \end{pmatrix}}{x,y,z \in \R} \subset \Sl_3 \R. \]
We obtain global coordinates $(x,y,z)$ from the obvious diffeomorphism with $\R^3$.  Then the group multiplication is
\[ (x,y,z) \cdot (z',y',z') = (x+x', y+y', z+z'+xy'). \]
There is a frame of left-invariant vector fields, 
\[ F_1 = \frac{\partial}{\partial x}, \quad F_2 = \frac{\partial}{\partial y} + x \frac{\partial}{\partial z}, \quad F_3 = \frac{\partial}{\partial z}, \]
and the only nontrivial Lie bracket relation is
\[ [F_1,F_2] = F_3. \]
The dual coframe is
\[ \theta^1 = dx, \quad \theta^2 = dy, \quad \theta^3 = dz - x dy. \]
A family of left-invariant metrics on $\Nil^3$ is given by
\[ \hat{g}(t) = A(t) \, \theta^1 \otimes \theta^1 + B(t) \, \theta^2 \otimes \theta^2 + C(t) \, \theta^3 \otimes \theta^3, \]
and the Ricci flow is the following system of ordinary differential equations:
\[ \frac{d}{dt} A = \frac{C}{B}, \quad \frac{d}{dt} B = \frac{C}{A}, \quad \frac{d}{dt} C = -\frac{C^2}{AB}. \]
It is well-known that the flow will preserve the diagonality of an initial metric, and the solution (with asymptotics) is
\begin{align*}
A(t) &= A_0 K^{-1/3} ( t + K )^{1/3}  \sim A_0 K^{-1/3} t^{1/3}, \\
B(t) &= B_0 K^{-1/3} ( t + K )^{1/3}  \sim B_0 K^{-1/3} t^{1/3}, \\
C(t) &= C_0 K^{1/3}  ( t + K )^{-1/3} \sim C_0 K^{1/3} t^{-1/3},
\end{align*}
for the constant
\[ K = \frac{A_0 B_0}{3 C_0}. \]
This solution exists for all time, but as $t \rightarrow \infty$, we see that $A,B \rightarrow \infty$, and $C \rightarrow 0$.  This is known as the ``pancake'' solution, as two directions are becoming more and more spread out, while the third is collapsing.

Calling the asymptotic solution $g(t)$, we see that the blowdown is
\begin{align*}
g_s(t)
&= A_0 K^{-1/3} s^{-2/3} t^{1/3} \, \theta^1 \otimes \theta^1 \\
&\qquad + B_0 K^{-1/3} s^{-2/3} t^{1/3} \, \theta^2 \otimes \theta^2 \\
&\qquad \qquad + C_0 K^{1/3} s^{-4/3} t^{-1/3} \, \theta^3 \otimes \theta^3.
\end{align*}

We now want to find the appropriate diffeomorphisms $\phi_s$.  Suppose that they are of the form
\[ \phi_s(x,y,z) = \big( \alpha(s) x, \beta(s) y, \gamma(s) z \big). \]
It is simple, then, to see that the functions
\begin{align*}
\alpha(s) &= (A_0 K^{-1/3})^{-1/2} s^{1/3} \\
\beta(s)  &= (B_0 K^{-1/3})^{-1/2} s^{1/3} \\
\gamma(s) = \alpha(s) \beta(s) &= (A_0 B_0 K^{-2/3})^{-1/2} s^{2/3}
\end{align*}
work as desired.  Thus,
\[ \phi_s^* g_s(t) = t^{1/3} \Big( \theta^1 \otimes \theta^1 + \theta^2 \otimes \theta^2 \Big) + \frac{1}{3} t^{-1/3} \,  \theta^3 \otimes \theta^3 = g_{\infty}(t), \]
and there is no need to take a limit.  A quick check shows that this is still a solution to Ricci flow, and that it satisfies
\[ g_{\infty}(t) = t \eta_t^* g_{\infty}(1) \]
for the diffeomorphisms
\[ \eta_t(x,y,z) = (t^{-1/3} x, t^{-1/3} y, t^{-2/3} z). \]
The metric $g_{\infty}(1)$ is the unique nilsoliton in dimension three, as seen in \cite{Lott2007}, \cite{Baird2007}, and \cite{Glick2008}.
\end{exmp}

\begin{rem}
Regarding the uniqueness of these diffeomorphisms in general, it is expected that if we have two families of diffeomorphisms, $\{\phi_s\}$ and $\{\psi\}_s$, that satisfy the above properties, then 
\[ \lim_{s \rightarrow \infty} \psi_s^{-1} \circ \phi_s \]
exists and is a diffeomorphism, even though $\phi_s$ and $\psi_s$ may not converge to diffeomorphisms individually.
\end{rem}

\section{Nilsolitons on Heisenberg groups}\label{sec-heis}

\subsection{The classical Heisenberg groups}\label{subsec-class-heis}

We now recall the construction and properties of the higher-dimensional, classical Heisenberg groups.  In terms of the framework outlined in \cite{Berndt1995}, these are simply connected Lie groups corresponding to \textit{generalized Heisenberg algebras} of the form $\mfr{n} = \mfr{v} \oplus \mfr{z}$, where $\mfr{z}$ is one-dimensional.  However, we will need a more explicit description.  Let $n$ be a positive integer, and set $N = 2n+1$.  The useful representation for us is 
\[ H_N \R = \setvbigl{ 
\begin{pmatrix}
1                  & \overrightarrow{a}^T & c \\
\overrightarrow{0} & I_n                  & \overrightarrow{b} \\
0                  & \overrightarrow{0}^T & 1 \end{pmatrix}}{\overrightarrow{a}, \overrightarrow{b} \in \R^n, c \in \R}  \subset \Sl_{n+2} \R, \]
where $I_n$ is the $n \times n$ identity matrix and $\overrightarrow{0} \in \R^n$ is the zero vector.  Group multiplication is again matrix multiplication:
\[ \begin{pmatrix}
1                  & \overrightarrow{a_1}^T & c_1 \\
\overrightarrow{0} & I_n                    & \overrightarrow{b_1} \\
0                  & \overrightarrow{0}^T   & 1 \end{pmatrix}
\begin{pmatrix}
1                  & \overrightarrow{a_2}^T & c_2 \\
\overrightarrow{0} & I_n                    & \overrightarrow{b_2} \\
0                  & \overrightarrow{0}^T   & 1 \end{pmatrix}
=
\begin{pmatrix}
1                  & \overrightarrow{a_1}^T + \overrightarrow{a_2}^T & c_1 + c_2 + \overrightarrow{a_1} \cdot \overrightarrow{b_2} \\
\overrightarrow{0} & I_n                                             & \overrightarrow{b_1} + \overrightarrow{b_2} \\
0                  & \overrightarrow{0}^T                            & 1 \end{pmatrix}, \]
or more briefly,
\[ (\overrightarrow{a_1},\overrightarrow{b_1},c_1) (\overrightarrow{a_2},\overrightarrow{b_2},c_2) 
= (\overrightarrow{a_1} + \overrightarrow{a_2},\overrightarrow{b_1} + \overrightarrow{b_2} ,c_1 + c_2 + \overrightarrow{a_1} \cdot \overrightarrow{b_2} ), \]
where $\cdot$ refers to the standard Euclidean inner product.  Clearly, this is a Lie group of dimension $N$.  It is easy to see that the Lie algebra of $H_N \R$ is 
\[ \mfr{h}_N \R = \setvbigl{ 
\begin{pmatrix}
0                  & \overrightarrow{X}^T & Z \\
\overrightarrow{0} & 0_n                  & \overrightarrow{Y} \\
0                  & \overrightarrow{0}^T & 0 \end{pmatrix}}{\overrightarrow{X}, \overrightarrow{Y} \in \R^n, Z \in \R}, \]
where $0_n$ is the $n \times n$ zero matrix.

If $\{e_i\}$ is the standard basis for $\R^n$, then we can describe a convenient basis for $\mfr{h}_N \R$.  Define 
\[ E_i = 
\begin{pmatrix}
0                  & e_i^T & 0 \\
\overrightarrow{0} & 0_n                  & \overrightarrow{0} \\
0                  & \overrightarrow{0}^T & 0 \end{pmatrix},
\quad 
E_{i+n} = 
\begin{pmatrix}
0                  & \overrightarrow{0}^T & 0 \\
\overrightarrow{0} & 0_n                  & e_i \\
0                  & \overrightarrow{0}^T & 0 \end{pmatrix},
\quad 
E_N = 
\begin{pmatrix}
0                  & \overrightarrow{0}^T & 1 \\
\overrightarrow{0} & 0_n                  & \overrightarrow{0} \\
0                  & \overrightarrow{0}^T & 0 \end{pmatrix}, \]
where $1 \leq i \leq n$, so that the basis is ordered as follows:
\[ \underbrace{E_1,\dots,E_n}_{E_{i}},\underbrace{E_{1+n},\dots,E_{2n}}_{E_{i+n}},E_N. \]
In what follows, lower case Roman indices will always range over $1,\dots,n$ (or sometimes $1,\dots,2n$) and Capital Roman indices (with the exception of $N$, which is fixed) will range over $1,\dots,N$.

The Lie bracket on $\mfr{h}_N \R$ is the usual matrix commutator, so the bracket relations are
\[ [E_i,E_j] = [E_i,E_N] = [E_{i+n},E_{j+n}] = [E_{i+n},E_N] = 0, \qquad [E_i,E_{j+n}] = \delta_{ij} E_N. \]
Thus, the only non-vanishing stucture constants are of the form
\begin{equation}\label{heis-sc}
c_{i,i+n}^N = 1.
\end{equation}

We have a diffeomorphism $H_N \R \cong \R^{N}$, which gives us coordinates:
\begin{equation}\label{heis-coords}
\begin{pmatrix}
1      & x^1    & \cdots & x^n    & x^N \\
0      & 1      & \cdots & 0      & x^{1+n}  \\
\vdots & \vdots & \ddots & \vdots & \vdots  \\
0      & 0      & \ddots & 1      & x^{2n}  \\
0      & 0      & \cdots & 0      & 1
\end{pmatrix} \longmapsto
(x^1,\dots,x^n,x^{1+n},\dots,x^{2n},x^N). 
\end{equation}
With respect to these coordinates, we can find a left-invariant frame with the same bracket relations as those above, and then find its coframe.  
 
\begin{lem}\label{heis-frame} With respect to the coordinates from (\ref{heis-coords}), $H_N \R$ has the following left-invariant frame $\{F_I\}$ and dual coframe $\{\theta^I\}$:
\begin{align*}
F_i     &= \partial_i,                       & \theta^i     &= dx^i,    \\
F_{i+n} &= \partial_{i+n} + x^i \partial_N,  & \theta^{i+n} &= dx^{i+n} \\
F_N     &= \partial_N,                       & \theta^N     &= dx^N - \sum_{k=1}^n x^k dx^{k+n}.
\end{align*}
The frame $\{F_I\}$ satisfies the same bracket relations as the basis $\{E_I\}$.
\end{lem}

\subsection{Computing the Ricci tensor}\label{subset-heis-ric}

We wish to analyze solutions to the Ricci flow 
\[ \frac{d}{dt} g = -2 \Rc, \]
starting at some initial metric $g_0$.  By Lemma \ref{heis-frame}, any one-parameter family of left-invariant metrics, and therefore any Ricci flow solution, $g(t)$ on $H_N \R$ can be written as
\[ g(t) = g_{IJ}(t) \, \theta^I \otimes \theta^J. \]
Analysis of these solutions requires a detailed understanding of the Ricci tensor.  For this we will use formula (\ref{riccomp}) from Appendix \ref{app-curv}, which utilizes the Lie algebra structure.  We break that equation apart as follows:
\begin{align*}
4 R_{IJ} 
&= \quad \underbrace{ \left[ 2 c_{KI}^P c_{JM}^Q + c_{KJ}^P c_{IM}^Q - c_{KM}^P c_{IJ}^Q \right] g^{KM} g_{PQ} }_{\langle 1 \rangle} \\
&\quad + \underbrace{ \left[ c_{JM}^P c_{PI}^Q g_{QK} - c_{JM}^P c_{PK}^Q g_{QI} + c_{KI}^P c_{PM}^Q g_{QJ} - c_{KI}^P c_{PJ}^Q g_{QM} \right] g^{KM} }_{\langle 2 \rangle} \\
&\quad + \underbrace{ \left[ (a_{KJ}^P + a_{JK}^P) (a_{IM}^Q + a_{MI}^Q) - (a_{KM}^P + a_{MK}^P) (a_{IJ}^Q + a_{JI}^Q) \right] g^{KM} g_{PQ} }_{\langle 3 \rangle}.
\end{align*}
The computations are relatively straight-forward, though lengthy, and so we omit them.  We simply remark that for each of the three pieces of $R_{IJ} = R_{IJ} \langle 1 \rangle + R_{IJ} \langle 2 \rangle + R_{IJ} \langle 3 \rangle$, one must consider six cases depending on index combinations:
\[ R_{ij}, \quad R_{i,j+n}, \quad R_{iN}, \quad R_{i+n,j+n}, \quad R_{i+n,N}, \quad R_{NN}. \]
We can see this structure in the following $N \times N$ matrix:
\[ R_{IJ} = \left( 
\begin{BMAT}{ccc.ccc.l}{ccc.ccc.c}
 &           &  & &             &  & \vphantom{R} \\
 & R_{ij}    &  & & R_{i,j+n}   &  & R_{iN} \\
 &           &  & &             &  & \vphantom{R} \\
 &           &  & &             &  & \vphantom{R} \\
 & R_{i,j+n} &  & & R_{i+n,j+n} &  & R_{i+n,N} \\
 &           &  & &             &  & \vphantom{R} \\
 & R_{i N}   &  & & R_{i+n,N}   &  & R_{NN} 
\end{BMAT}
\right). \]
If we set
\[ \Sigma = \sum_{k,m=1}^n g^{km} g^{k+n,m+n} - \sum_{k=1}^n \sum_{m = n+1}^{2n} g^{k m} g^{k+n,m -n}, \]
then the components of the Ricci tensor are
\begin{align*}
R_{ij}      &= -\frac{1}{2} g^{i+n,j+n} g_{NN} + \frac{1}{2} g_{iN} g_{jN} \Sigma, \\
R_{i,j+n}   &= \frac{1}{2} g^{i+n,j} g_{NN} + \frac{1}{2} g_{iN} g_{j+n,N} \Sigma, \\
R_{iN}      &= \frac{1}{2} g_{iN} g_{NN} \Sigma, \\
R_{i+n,j+n} &= -\frac{1}{2} g^{ij} g_{NN} + \frac{1}{2} g_{i+n,j+n} g_{\beta N} \Sigma, \\
R_{i+n,N}   &= \frac{1}{2} g_{i+n,N} g_{NN} \Sigma, \\
R_{NN}      &= \frac{1}{2} g_{NN}^2 \Sigma. 
\end{align*}

\subsection{The Ricci flow}\label{subset-heis-rf}

Due to the complexity of the inverse of $g$, solving the Ricci flow system for arbitary initial data is intractable.  Instead, we assume that we have diagonal initial data, and show that the flow preserves diagonality.  So, if we assume that $g_{IJ} = g^{IJ} = 0$ for all $I \ne J$, then we claim that $R_{IJ} = 0$ as well.  From now on, we only use single subscripts for the metric components: $g_1,\dots,g_N$.  

We first note that when $g$ is diagonal, we have
\begin{equation}\label{diag-Sigma}
\Sigma = \sum_{k=1}^n g^{kk} g^{k+n,k+n} = \sum_{k=1}^n \frac{1}{g_k g_{k+n}}.
\end{equation}

Then we have
\begin{align*}
R_{ij}      &= \begin{cases} -\frac{1}{2} g^{i+n} g_N & \text{ if } i = j \\
				                         0 & \text{ if } i \ne j \end{cases}, \\
R_{i,j+n}   &= 0, \\
R_{iN}      &= 0, \\
R_{i+n,j+n} &= \begin{cases} -\frac{1}{2} g^{i} g_N & \text{ if } i = j \\
						  0 & \text{ if } i \ne j \end{cases}, \\
R_{i+n,N}   &= 0, \\
R_{NN}      &= \frac{1}{2} g_N^2 \Sigma. 
\end{align*}
This means that the natural basis for $\mfr{h}_N \R$ (or the frame for $H_N \R$ from Lemma \ref{heis-frame}) is stably Ricci-diagonal.  In other words, the Ricci tensor stays diagonal under the flow, and we have some hope to understand the behavior of the Ricci flow system: 
\begin{align}
\frac{d}{dt}g_i        &= \frac{g_N}{g_{i+n}} \label{rf-i} \\
\frac{d}{dt}g_{i+n} &= \frac{g_N}{g_i} \label{rf-in} \\
\frac{d}{dt}g_N        &= -g_N^2 \sum_{k=1}^n \frac{1}{g_k g_{k+n}} = -g_N^2 \Sigma \label{rf-N} 
\end{align}
for $i=1,\dots,n$ and $N=2n+1$.

\begin{rem}
It is possible to find an explicit Ricci flow solution in some cases.  For example, make following ansatz.  Let $X(t) = t+K$, where $K$ is some constant depending on the initial data.  This means $X(0) = K$ and $X'(t) = 1$.  Then we look for solutions of the form
\[ g_i(t) = \gamma_i X^{1/n+2}, \quad g_{\alpha}(t) = \gamma_{\alpha} X^{1/n+2}, \quad g_N(t) = \gamma_N X^{-n/n+2}. \]
However, when solving for $K$, constraints on the initial data appear.  In particular, a solution of this form requires initial data to come from an $(n+1)$-paramater family of diagonal metrics.
\end{rem}

Here we again note that there is indeed a Ricci soliton on $H_N \R$.  This follows from a theorem of Lauret.

\begin{thm}[\cite{Lauret2001}]\label{lauret-thm} A homogeneous nilmanifold $(N,g)$ with corresponding metric Lie algebra $(\mfr{n},\langle \cdot,\cdot \rangle_{\mfr{n}})$ is a Ricci soliton if and only if $(\mfr{n},\langle \cdot,\cdot \rangle_{\mfr{n}})$ admits a metric solvable extension $(\mfr{s} = \mfr{a} \oplus \mfr{n}, \langle \cdot,\cdot \rangle_{\mfr{s}})$, with $\mfr{a}$ Abelian, whose corresponding solvmanifold $(S,\tilde{g})$ is Einstein.
\end{thm}

The simply connected Lie group corresponding to the Lie algebra $\mfr{s} = \mfr{h}_N \R \oplus \mfr{\R}$ is an example of a \textit{Damek-Ricci space}.  These are known to be Einstein manifolds, and their Lie algebras are metric solvable extensions as in the theorem; see \cite{Berndt1995} for details.  Therefore, there is a left-invariant metric $g$ such that $(H_N \R,g)$ is a Ricci nilsoliton.  

As mentioned in the introduction, this metric is unique up to scaling and isometry.  We will describe it explicitly.

\subsection{Asymptotics of general solutions}
Now we consider the behavior of arbitrary diagonal solutions of Ricci flow.  From this we will obtain the nilsoliton using the blowdown method.

Assume that $g_1,\dots,g_{2n}, g_N$ solve the Ricci flow.  As diagonal components of a metric, they are positive functions of $t$.  We can use (\ref{rf-i}), (\ref{rf-in}), and (\ref{rf-N}) to see that
\[ \frac{d}{dt} \frac{g_i}{g_{i+n}} = \frac{d}{dt} \frac{g_{i+n}}{g_i} = 0, \]
\[ \frac{d}{dt} g_1 \cdots g_n g_N = \frac{d}{dt} g_{1+n} \dots g_{2n} g_N = 0. \]
This means these quantities are conserved (i.e., constant), so we set
\begin{align*}
A_i     &= \frac{g_i}{g_{i+n}} = \frac{g_i(0)}{g_{i+n}(0)}, \\
B_{i+n} &= \frac{g_{i+n}}{g_i} = \frac{g_{i+n}(0)}{g_i(0)}, \\
C_1     &= g_1 \cdots g_n g_N = g_1(0) \cdots g_n(0) g_N(0), \\
C_2     &= g_{1+n} \dots g_{2n} g_N = g_{1+n}(0) \dots g_{2n}(0) g_N(0),
\end{align*}
where $1 \leq i \leq n$.  Note that $A_i B_{i+n} = 1$, and that
\begin{equation}\label{c1c2-rel}
\frac{C_1}{A_1 \cdots A_n} = C_2, \quad \frac{C_2}{B_{1+n} \cdots B_{2n}} = C_1. 
\end{equation}

We rewrite the Ricci flow equation for $g_i$:
\begin{equation}\label{rf-ii}
\frac{d}{dt} g_i = \frac{g_N}{g_{i+n}} \frac{g_i}{g_i} = A_i \frac{g_N}{g_i},
\end{equation}
and similarly
\begin{equation}\label{rf-i2}
\frac{d}{dt} g_i^2 = 2 g_i \frac{g_N}{g_{i+n}} = 2 A_i g_N,
\end{equation}
which can be solved by integrating. 

Note that (\ref{rf-i}) implies that $g_i$ is an increasing function, so $\Sigma$ is positive and decreasing by (\ref{diag-Sigma}).  Then equation (\ref{rf-N}) implies that $g_N$ is a decreasing function, and since it is positive we have 
\[ \frac{d}{dt} g_N = -g_N^2 \Sigma \geq -\Sigma(0) g_N^2, \]
and this implies
\begin{equation}\label{gN-sol}
g_N(t) \geq \frac{1}{g_N(0)^{-1} + \Sigma(0)t}.
\end{equation}

If we set $G_N(t) = \int_1^t g_N(r) \, dr$, then this is a positive, increasing function.  By (\ref{gN-sol}), we see that
\[ \lim_{t \rightarrow \infty} G_N(t) = \lim_{t \rightarrow \infty} \int_1^t g_N(r) \, dr \geq \lim_{t \rightarrow \infty} \int_1^t \frac{dr}{g_N(0)^{-1} + \Sigma(0) r} = \infty, \]
so $G_N(t) \rightarrow \infty$ as $t \rightarrow \infty$.  Using (\ref{rf-i2}) we have
\[ g_i(t)^2 = g_i(0)^2 + 2 A_i G_N(t). \]
If $1 \leq i \ne j \leq n$, we use this to obtain
\[ \lim_{t \rightarrow \infty} \frac{g_i^2}{g_j^2} = \lim_{t \rightarrow \infty} \frac{g_i(0)^2 + 2 A_i G_N(t)}{g_i(0)^2 + 2 A_j G_N(t)} = \frac{A_i}{A_j}. \]
This implies that
\begin{equation}\label{gi-asym1}
g_i \sim \sqrt{\frac{A_i}{A_j}} g_j, \quad g_{i+n} \sim \sqrt{\frac{B_{i+n}}{B_{j+n}}} g_{j+n}.
\end{equation}
Since $C_1$ is conserved, for each fixed $1 \leq i \leq n$, we have
\begin{align}
g_N 
&= \frac{C_1}{g_1 \cdots g_n} \nonumber \\
&\sim \frac{1}{ \sqrt{\frac{A_1}{A_i}} g_i \cdots \sqrt{\frac{A_i}{A_i}} g_i \cdots \sqrt{\frac{A_n}{A_i}} g_i} \nonumber \\
&= \sqrt{\frac{A_i^n}{A_1 \cdots A_n}} g_i^{-n} \nonumber \\ 
&= \sqrt{A_i^n C_1 C_2} g_i^{-n}, \label{gN-asym1}
\end{align}
by (\ref{gi-asym1}).  With reference to (\ref{rf-ii}), this gives
\[ A_i \frac{g_N}{g_i} \sim \sqrt{A_i^{n+2} C_1 C_2} g_i^{-(n+1)}. \]
We would like to see that the solution $\tilde{g}_i$ to the equation
\begin{equation}\label{gi-eqn}
\frac{d}{dt} \tilde{g}_i = \sqrt{A_i^{n+2} C_1 C_2} \tilde{g}_i^{-(n+1)} 
\end{equation}
is asymptotically equivalent to $g_i$.  For this we need a basic lemma.

\begin{lem}
Suppose that $u(t)$ is a solution to the ordinary differential equation 
\[ \frac{d}{dt} u = c, \quad u(0) = u_0, \]
where $c>0$, and that $v(t)$ is a solution to the asymptotic equation
\[ \frac{d}{dt} v = c(1 + \epsilon(t)), \quad v(0) = v_0, \]
where $\epsilon(t) \rightarrow 0$ as $t \rightarrow \infty$.  Then $u/v \rightarrow 1$ as $t \rightarrow \infty$.  That is, $u \sim v$.
\end{lem}

\begin{proof}
We can solve both equations by integrating:
\[ u(t) = u_0 + c \int_0^t \, ds = t \left( \frac{u_0}{t} + c \right), \]
\[ v(t) = v_0 + c \int_0^t (1+\epsilon(s)) \, ds = t \left( \frac{u_0}{t} + c + \frac{c}{t} \int_0^t \epsilon(s) \, ds \right). \]
To analyze the ratio $u/v$, we must know the behavior of the integral term in $v$.  Note that, as a positive increasing function, 
\[ \int_0^t |\epsilon(s)| \, ds \longrightarrow L \in (0,\infty] \]
as $t \rightarrow \infty$.  If $L < \infty$, then
\[ \left| \lim_{t \rightarrow \infty} \frac{c}{t} \int_0^t \epsilon(s) \, ds \right| \leq c \lim_{t \rightarrow \infty} \frac{1}{t} \int_0^t |\epsilon(s)| \, ds \leq cL \lim_{t \rightarrow \infty} \frac{1}{t} = 0. \]
On the other hand, if $L = \infty$, then
\[ \left| \lim_{t \rightarrow \infty} \frac{c}{t} \int_0^t \epsilon(s) \, ds \right| \leq c \lim_{t \rightarrow \infty} \frac{1}{t} \int_0^t |\epsilon(s)| \, ds \stackrel{LH}{=} c \lim_{t \rightarrow \infty} \frac{|\epsilon(t)|}{1} = 0. \]
This means 
\[ \lim_{t \rightarrow \infty} \frac{u}{v} = \lim_{t \rightarrow \infty} \frac{\frac{u_0}{t} + c}{ \frac{u_0}{t} + c + \frac{c}{t} \int_0^t \epsilon(s) \, ds} = 1, \]
by the squeeze theorem.
\end{proof}

Note that equation (\ref{gi-eqn}) is equivalent to 
\[ \frac{d}{dt} \tilde{g}_i^{n+2} = (n+2) \sqrt{A_i^{n+2} C_1 C_2}, \]
and equation (\ref{rf-ii}) is equivalent to
\[ \frac{d}{dt} g_i^{n+2} = (n+2) A_i g_i^n g_N. \]
By (\ref{gN-asym1}), the right sides of these equations are asymptotically equivalent.  Now, taking $u = \tilde{g}_i^{n+2}$ and $v = g_i^{n+2}$ in the Lemma, we see that $g_i \sim \tilde{g}_i$.  

Equation (\ref{gi-eqn}) has an explicit solution:
\begin{align}
\tilde{g}_i 
&= \left( (n+2) \sqrt{A_i^{n+2} C_1 C_2} \right)^{1/n+2} t^{1/n+2} \nonumber \\
&= (n+2)^{1/n+2} \sqrt{A_i} (C_1 C_2)^{1/2(n+2)} t^{1/n+2}. \label{gi-asym2}
\end{align}
We can plug this into (\ref{gN-asym1}) to obtain
\begin{align}
g_N 
&= \frac{C_1}{g_1 \cdots g_n} \nonumber \\
&\sim \frac{C_1}{(n+2)^{n/n+2} \sqrt{A_1 \cdots A_n} (C_1 C_2)^{n/2(n+2)}} t^{-n/n+2} \nonumber \\
&= (n+2)^{-n/n+2} (C_1 C_2)^{1/n+2} t^{-n/n+2}, \label{gN-asym2}
\end{align}
which we call $\tilde{g}_N$.  Note that this is independent of $i$.

We can repeat these calculations starting with $g_N = C_2/g_{n+1} \cdots g_{2n}$ to obtain 
\[ \tilde{g}_{i+n} = (n+2)^{1/n+2} \sqrt{B_{i+n}} (C_1 C_2)^{1/2(n+2)} t^{1/n+2}. \]
If we plug this back into $g_N = C_2/g_{1+n} \cdots g_{2n}$, then we get the same result for $g_N$ that we found in equation (\ref{gN-asym2}).  Putting everything together, we have the following result.

\begin{thm*}[\ref{thm:heis}(a)]
If $g_0$ is a diagonal left-invariant metric on $H_N \R$, then the solution $g(t)$ of Ricci flow, with $g(0) = g_0$, has the following asymptotic behavior:
\begin{align*}
g_i     &\sim (n+2)^{1/n+2} \sqrt{A_i} C^{1/2(n+2)} t^{1/n+2}, \\
g_{i+n} &\sim (n+2)^{1/n+2} \sqrt{B_{i+n}} C^{1/2(n+2)} t^{1/n+2}, \\
g_{N}   &\sim (n+2)^{-n/n+2} C^{1/n+2} t^{-n/n+2},
\end{align*}
where
\[ A_i = \frac{g_i(0)}{g_{i+n}(0)}, \quad
B_{i+n} = \frac{g_{i+n}(0)}{g_i(0)}, \quad
C       = g_1(0) \cdots g_{2n}(0) g_N(0)^2. \]
\end{thm*}

\begin{rem} These asymptotics coincide with the case $n=1$ from Example \ref{nil3}.
\end{rem}

\subsection{The nilsoliton}\label{subsec-heis-nil}

Writing $g(t)$ for the asymptotic solution of Theorem \ref{thm:heis}(a), we now use the blowdown procedure of Section \ref{sec-blowdown} to obtain the soliton metric.  The components of $g_s(t)$ are
\begin{align*}
(g_s(t))_i     &= (n+2)^{1/n+2} A_i^{1/2} C^{1/2(n+2)} s^{-(n+1)/n+2} t^{1/n+2}, \\
(g_s(t))_{i+n} &= (n+2)^{1/n+2} B_{i+n}^{1/2} C^{1/2(n+2)} s^{-(n+1)/n+2} t^{1/n+2}, \\
(g_s(t))_N     &= (n+2)^{-n/n+2} C^{1/n+2} t^{-n/n+2}.
\end{align*}
Using the coordinates and coframe from Subsection \ref{subsec-class-heis}, we seek diffeomorphisms $\phi_s$ such that $\phi_s^* g_s(t)$ is a metric for all $s$, and 
\[ \lim_{s \rightarrow \infty} \phi_s^* g_s(t) \]
exists.  Suppose that $\phi_s$ is of the form
\[ \phi_s(x^1,\dots,x^N) = \left( \alpha^1(s) x^1, \dots, \alpha^N(s) x^N \right) \]
for some functions $\alpha^i(s)$.  Then we have $\phi_s^* \theta^i = \alpha^i(s) \, \theta^i$ for $1 \leq i \leq 2n$. When $1 \leq i \leq n$, setting 
\[ \alpha^i(s) = (n+2)^{-1/2(n+2)} A_i^{-1/4} C^{-1/4(n+2)} s^{(n+1)/2(n+2)} \]
gives, for fixed $i$, 
\[ \phi_s^* \left( (g_s(t))_i \, \theta^i \otimes \theta^i \right) = \alpha^i(s)^2 (g_s(t))_i \, \theta^i \otimes \theta^i = t^{1/n+2} \, \theta^i \otimes \theta^i. \]
Next, note that 
\[ \alpha^i(s) \alpha^{i+n}(s) = (n+2)^{-1/n+2} C^{-1/2(n+2)} s^{n+1/n+2}, \]
and this does not depend on $i$.  Therefore, if we set $\alpha^N(s) = \alpha^i(s) \alpha^{i+n}(s)$, then we have
\begin{align*}
\phi_s^* \theta^N 
&= \alpha^i(s) \alpha^{i+n}(s) \, dx^N - \sum_{i=1}^n \alpha^i(s) \alpha^{i+n}(s) x^i \, dx^{i+n} \\
&= \alpha^N(s) \, \theta^N, 
\end{align*}
and so
\[ \phi_s^* \left( (g_s(t))_N \, \theta^N \otimes \theta^N \right) = \frac{1}{n+2} t^{-n/n+2} \, \theta^N \otimes \theta^N. \]

We have a limit metric
\[ g_{\infty}(t) = \phi_s^* g_s(t) = t^{1/n+2} \Big( \theta^1 \otimes \theta^1 + \cdots + \theta^{2n} \otimes \theta^{2n} \Big) + \frac{1}{n+2} t^{-n/n+2} \, \theta^N \otimes \theta^N, \]
and to verify that it is a soliton, we seek diffeomorphisms $\{\eta_t\}$ such that $g_{\infty}(t)$ satisfies
\[ g_{\infty}(t) = t \eta_t^* g_{\infty}(1). \]
For some numbers $a$ and $b$, suppose that the diffeomorphisms are of the form
\[ \eta_t(x^1,\dots,x^{2n},x^N) = (t^a x^1,\dots,t^a x^{2n},t^b x^N). \]
Then for $1 \leq i \leq 2n$, we have $\eta_t^* \theta^i = t^a \theta^i$ and, if $b = 2a$, $\eta_t^* \theta^N = t^b \theta^N$.  This means
\[ t \eta_t^* g(1) = t^{2a + 1} \Big( \theta^1 \otimes \theta^1 + \cdots + \theta^{2n} \otimes \theta^{2n} \Big) + \frac{1}{n+2} t^{2b + 1} \, \theta^N \otimes \theta^N. \]
For this to equal $g(t)$, we must have
\[ \frac{1}{n+2} = 2 a + 1, \qquad -\frac{n}{n+2} = 2b + 1, \]
which implies 
\[ a = -\frac{1}{2} \frac{n+1}{n+2}, \qquad b = 2a = - \frac{n+1}{n+2}. \]
Thus, $g(t)$ is an expanding Ricci soliton with respect to the diffeomorphisms
\[ \eta_t(x^1,\dots,x^{2n},x^N) = (t^{-\frac{1}{2} \frac{n+1}{n+2}} x^1,\dots,t^{-\frac{1}{2} \frac{n+1}{n+2}} x^{2n},t^{-\frac{n+1}{n+2}} x^N). \]
To summarize, we have another result.

\begin{thm*}[\ref{thm:heis} (b)]
Let $H_N \R$ have coordinates $(x^i)$ as in (\ref{heis-coords}) and coframe as in Lemma \ref{heis-frame}.  Let $g(t)$ be any solution to Ricci flow on $H_N \R$ with diagonal initial data.  For the diffeomorphisms $\{\phi_s\}$ defined as above, we have
\begin{align*}
\lim_{s \rightarrow \infty} \frac{1}{s} \phi_s^* g(st) 
&= t^{1/n+2} \Big( \theta^1 \otimes \theta^1 + \cdots + \theta^{2n} \otimes \theta^{2n} \Big) + \frac{1}{n+2} t^{-n/n+2} \, \theta^N \otimes \theta^N \\
&= g_{\infty}(t).
\end{align*}
The metric $g_{\infty}(1)$ is a nilsoliton with respect to the diffeomorphisms
\[ \eta_t(x^1,\dots,x^{2n},x^N) = (t^{-\frac{1}{2} \frac{n+1}{n+2}} x^1,\dots,t^{-\frac{1}{2} \frac{n+1}{n+2}} x^{2n},t^{-\frac{n+1}{n+2}} x^N). \]
\end{thm*}

The behavior here is analagous to the ``pancake'' effect mentioned in Example \ref{nil3}.  The first $2n$ directions become more and more spread out, while the last direction collapses.  More precisely, there is Gromov-Hausdorff convergence to $(\R^{2n},g_{\mathrm{can}})$.

\begin{rem}
The diffeomorphisms $\phi_s$ and $\eta_t$ here, and those in Example \ref{nil3}, are actually group automorphisms.  Compare with \cite{Lott2007}, Remark 3.1 and Section 4.
\end{rem}

\begin{rem}
Looking at the three-dimensional nilsoliton, one can extrapolate with the following ansatz:
 \[ g_i = g_{i+n} = t^a, \quad g_N = c t^b, \]
for some numbers $a,b$ and $c$.  Using the Ricci flow equations, it is easy to obtain
\[ a = \frac{1}{n+2}, \quad b = \frac{-n}{n+2}, \quad c = \frac{1}{n+2}. \]
Thus,
\[ g_i(t) = g_{i+n}(t) = t^{1/n+2}, \quad g_N(t) = \frac{1}{n+2} t^{-n/n+2}, \]
which is the nilsoliton $g_{\infty}$ above.  This does not provide any information about behavior of general solutions, however.
\end{rem}

\subsection{The groupoid interpretation}\label{subsec-heis-grpd}

In \cite{Lott2007} and \cite{Lott2010}, Lott initiated the use of Riemannian groupoids in understanding the notion of convergence under Ricci flow.  One motivating issue is that, as in the case of $\Nil^3$, the limit of a Ricci flow solution $(M,g(t))$ as $t \rightarrow \infty$ may not be an object of the same dimension (i.e., it may collapse).  This means some data has been lost in the process of taking the limit.  The groupoid formalism provides a way to keep track of all such data (e.g., the limiting object has the same dimension as $M$), and to provide a picture of the limiting behavior that is similar to, but more convenient than, the usual Gromov-Hausdorff notion of convergence.  One may consult \cite{Lott2007} and \cite{Glick2008} for background on Riemannian groupoids, or the books \cite{Mackenzie2005}, \cite{Moerdijk2003} for a more general introduction to groupoids. 

Our analysis here follows the examples found in \cite{Glick2008}, which give concrete pictures of collapse.  Here is the basic idea, tailored to our present context.  In order to understand the collapse under Ricci flow of certain compact, locally homogenous manifolds arising as quotients of $H_N \R$, we replace such a manifold $(M=H_N \R/\Gamma,g)$ by its representation as a Riemannian ``action'' groupoid, $(H_N \R \rtimes \Gamma,\tilde{g})$.  Also called a ``cross-product'' groupoid, this is an object whose orbit space is $M$.  Here, 
\[ \fnl{\pi}{(H_N \R,\tilde{g})}{(M,g)} \]
is the universal cover with induced metric, and $\Gamma \subset H_N \R$ is a discrete, cocompact subgroup that can be interpreted in several ways.  It is the fundamental group $\pi_1(M,m_0)$, the group of deck transformation of the cover, or a group of isometries acting transitively on $(H_N \R,\tilde{g})$.  In any case, it acts by left translation on $H_N \R$. 

If $g(t)$ is a Ricci flow solution on $M$, then we are considering a solution $\tilde{g}(t)$ on $H_N \R$.  By the prevous section, the blowdown technique provides a sequence $\phi_{s} \tilde{g}_{s}(t)$ of metrics converging to a metric $\tilde{g}_{\infty}(t)$, where $\tilde{g}_{\infty}(1)$ is a soliton.  To understand the limiting behavior as $s \rightarrow \infty$, we now consider
\[ (H_N \R \rtimes \Gamma_{s}, \phi_{s} \tilde{g}_{s}(t)). \]
Note that the subgroup $\Gamma_{s}$ acting on $H_N \R$ depends on $s$, since the metric is changing.  If, in the limit, this sequence of discrete subgroups converges to a continuous subgroup, then there is collapse.  Therefore, we must understand how these subgroups evolve.

Recall that the blowdown metrics $\tilde{g}_s(t)$ are obtained using diffeomorphisms 
\[ \phi_s(x^1,\dots,x^N) = (\alpha^1(s) x^1, \dots, \alpha^N(s) x^N). \]
(The explicit forms of the $\alpha$'s are not imporant here.)  Then the limit is
\[ \tilde{g}_{\infty}(t) = \phi_s^* g_s(t) = t^{1/n+2} \Big( \theta^1 \otimes \theta^1 + \cdots + \theta^{2n} \otimes \theta^{2n} \Big) + \frac{1}{n+2} t^{-n/n+2} \theta^N \otimes \theta^N. \]

Without loss of generality, after change of coordinates we can take $\Gamma_s$ to be an integer lattice.  Therefore, write elements of $\Gamma_s$ as
\[ h_z(s) = h_{z^1(s),\dots,z^N(s)} = \left( z^1(s),\dots,z^N(s) \right), \]
with $z^i(s) \in \Z$.  These isometries act on $(H_N \R,\tilde{g}_s(t))$ by left translation and, as deck transformations, they pull back by conjugation. Therefore, 
\begin{align*}
&\phi_x^* h_z (x^1,\dots,x^{2n},x^N) \\
&= \phi_x^{-1} h_z \phi_s (x^1,\dots,x^{2n},x^N) \\
%&= \phi_x^{-1} h_z (\alpha^1 x^1, \dots, \alpha^{2n} x^{2n}, \alpha^N x^N) \\
%&= \phi_x^{-1} (\alpha^1 x^1 + z^1, \dots, \alpha^{2n} x^{2n} + z^{2n}, \\
%&\qquad \alpha^N x^N + z^N + z^1 \alpha^{n+1} x^{n+1} + \cdots z^n \alpha^{2n} x^{2n}) \\
%&= ( x^1 + (\alpha^1)^{-1} z^1, \dots, x^{2n} + (\alpha^{2n})^{-1} z^{2n}, \\
%&\qquad x^N + (\alpha^N)^{-1} z^N + z^1 (\alpha^N)^{-1} \alpha^{n+1} x^{n+1} + \cdots + z^n (\alpha^N)^{-1} \alpha^{2n} x^{2n} ) \\
&= \left( x^1 + \frac{z^1(s)}{\alpha^1(s)}, \dots, x^{2n} + \frac{z^{2n}(s)}{\alpha^{2n}(s)}, x^N + \frac{z^N(s)}{\alpha^N(s)} + \frac{z^1(s)}{\alpha^1(s)} x^{n+1} + \cdots + \frac{z^n(s)}{\alpha^n(s)} x^{2n} \right),
\end{align*}
using the component-wise form of the group multiplication.  

It is a basic fact that, given any strictly increasing sequence $\{\sigma_j\}$ with $\sigma_j \rightarrow \infty$ as $j \rightarrow \infty$, and any $u \in \R$, there is some sequence of integers $\{\tau_j\}$ such that $\tau_j/\sigma_j \rightarrow u$.  Indeed, take $\tau_j = \lfloor \sigma_j u \rfloor$.

Therfore, consider any strictly increasing sequence $\{s_j\}$ with $s_j \rightarrow \infty$ as $j \rightarrow \infty$.  The sequences $\{\alpha^I(s_j)\}$ are also strictly increasing.  Then given any real numbers $u^1,\dots,u^N$, we may choose $z^i(s_j) \in \Gamma_{s_j}$ such that 
\[ \lim_{j \rightarrow \infty} \frac{z^i(s_j)}{\alpha^i(s_j)} = u^i, \quad \lim_{j \rightarrow \infty} \frac{z^{i+n}(s_j)}{\alpha^{i+n}(s_j)} = u^{i+n}, \quad \lim_{j \rightarrow \infty} \frac{z^N(s_j)}{\alpha^N(s_j)} = u^N. \]
This means that as $j \rightarrow \infty$, the isometries $\phi_{s_j}^* h_z$ converge to isometries $h_u$ of $\tilde{g}_{\infty}(t)$ that act on $H_N \R$ as follows:
\[ h_u(x^1,\dots,x^{2n},x^N) = (x^1 + u^1,\dots,x^{2n} + u^{2n}, x^N + u^N + u^1 x^{n+1} + \cdots + u^n x^{2n}). \]
The $u^i$ were arbitary real numbers, so every element of $H_N \R$ is attained this way.  This means $\Gamma_{s_j}$ converges to a continuous group: the entire group $H_N \R$.

We conclude that 
\[ \lim_{j \rightarrow \infty} (H_N \rtimes \Gamma_{s_j},\phi_{s_j}^* \tilde{g}_{s_j}(t)) = (H_N \R \rtimes H_N \R,\tilde{g}_{\infty}(t)) \]
as Riemannian groupoids.  There is maximal collapsing, as the orbit space of the groupoid $H_N \R \rtimes H_N \R$ is a point.  This is the same behavior seen in the three-dimensional case. 

\begin{rem}
Note that this is a different description than the ``pancake'' model described earlier, which occurs as $t \rightarrow \infty$.  The model here illustrates collapse as the metrics converge to the actual soliton metric.
\end{rem}

\section{Nilsolitons on spaces of unitriangular matrices}\label{sec-uni}

\subsection{Unitriangular matrices}\label{subsec-uni}

Let $\UT_n \R \subset \Sl_n\R$ denote the collection of real, unitriangular $n \times n$ matrices under matrix multiplication.  These are matrices with $1$ on the diagonal and $0$ below.  This is a Lie group of dimension $N = \binom{n}{2} = n(n-1)/2$, and $\UT_n \R \cong \R^N$.  These groups are nilpotent, and are in some sense ``model'' nilpotent Lie groups.  Indeed, it is a consequence of Engel's theorem that every simply connected nilpotent Lie group is a subgroup of $\UT_n \R$ for some $n$.

The Lie algebra $\mfr{ut}_n \R$ of $\UT_n \R$ consists of upper-triangular matrices with 0 on the diagonal.  It has a basis 
\[ \mcl{B}_n = \{B_{ij}\}_{1 \leq i < j \leq n}, \]
where $B_{ij}$ is the $n \times n$ matrix such that that 
\[ (B_{ij})_{pq} = \delta_{ip} \delta_{jq}. \]
In other words, $B_{ij}$ is the matrix with 1 in the $(i,j)$ component, and zero elsewhere.  

This Lie algebra inherits the Lie bracket from $\mfr{gl}_n \R$.  To describe the bracket, note that if $i < j$ and $k < l$, then
\[ (B_{ij} B_{kl})_{pq} = \sum_r (B_{ij})_{pr} (B_{kl})_{rq} = \sum_r \delta_{ip} \delta_{jr} \delta_{kr} \delta_{lq} = \delta_{ip} \delta_{lq} \delta_{jk} = \delta_{jk} (B_{il})_{pq}, \]
which implies
\[ [B_{ij},B_{kl}] = \delta_{jk} B_{il} - \delta_{il} B_{kj}, \]
and so the structure constants are
\begin{equation}\label{str-consts}
c_{ij,kl}^{pq} = \delta_{ip} \delta_{lq} \delta_{jk} - \delta_{kp} \delta_{jq} \delta_{il}.
\end{equation}

Any diffeomorphism $\UT_n \R \cong \R^{N}$ gives us coordinates, so let us take
\begin{equation}\label{uni-coords}
\begin{pmatrix} 
1      & x^{12} & x^{13} & \cdots & x^{1,n-1} & x^{1n} \\ 
0      & 1      & x^{23} & \cdots & x^{2,n-1} & x^{2n} \\
0      & 0      & 1      & \cdots & x^{3,n-1} & x^{3n} \\
\vdots & \vdots & \vdots & \ddots & \vdots    & \vdots \\
0      & 0      & 0      & \cdots & 1         & x^{n-1,n} \\
0      & 0      & 0      & \cdots & 0         & 1
\end{pmatrix} \longmapsto
(x^{12},x^{13},\dots,x^{n-1,n}). 
\end{equation}
With respect to these coordinates, we can find a left-invariant frame with the same bracket relations as those above, and then find its coframe.  If $a = (a_{ij})$ and $b = (b_{ij})$ are elements of $\UT_n \R$, the multiplication rule is
\[ x^{ij}(a \cdot b) = x^{ij}(a) + x^{ij}(b) + \sum_{i < k < j} x^{ik}(a) x^{kj}(b). \]
 
\begin{lem}\label{uni-frame} With respect to the coordinates from (\ref{uni-coords}), the space $\UT_n \R$ has the following left-invariant frame $\{F_{ij}\}$ and dual coframe $\{\theta^{ij}\}$:
\begin{align*}
F_i         &= \partial_{ij} + \sum_{k<i} x^{ki} \partial_{kj}, \\
\theta^{ij} &= dx^{ij} - \sum_{i<p<j}( 2x^{ip}-\Theta^{ij}_p ) \, dx^{pj}, 
\end{align*}
where 
\[ \Theta_p^{ij} = \sum_{k=0}^p \sum_{i < r_1 < \cdots < r_k < p} x^{i r_1} x^{r_1 r_2} \cdots x^{r_k p}, \]
and the inner sum ranges over all ordered subsets of $\{i+1,i+2\dots,p-1\}$ of size $k$.  The frame $\{F_{ij}\}$ satisfies the same bracket relations as the basis $\{B_{ij}\}$ above.
\end{lem}

\subsection{Computing the Ricci tensor}\label{subsec-uni-ric}

Our goal is to analyze solutions of Ricci flow on $\UT_n \R$.  By Lemma \ref{uni-frame}, such metrics $g(t)$ can be written as
\[ g(t) = g_{ij,kl}(t) \, \theta^{ij} \otimes \theta^{kl}. \]
Once again, our analysis requires us to understand the Ricci tensor, and equation (\ref{riccomp}) still applies.  In terms of $\UT_n \R$, where double indices are needed, we can rewrite it as
\begin{align}
&4 R_{ij,kl} = \label{ricci} \\
&\quad \underbrace{ \left[ 2 c_{pq,ij}^{tu} c_{kl,rs}^{vw} + c_{pq,kl}^{tu} c_{ij,rs}^{vw} - c_{pq,rs}^{tu} c_{ij,kl}^{vw} \right] g^{pq,rs} g_{tu,vw} }_{\langle 1 \rangle}  \nonumber\\
&+ \underbrace{ \left[ c_{kl,rs}^{tu} c_{tu,ij}^{vw} g_{vw,pq} - c_{kl,rs}^{tu} c_{tu,pq}^{vw} g_{vw,ij} + c_{pq,ij}^{tu} c_{tu,rs}^{vw} g_{vw,kl} - c_{pq,ij}^{tu} c_{tu,kl}^{vw} g_{vw,rs} \right] g^{pq,rs} }_{\langle 2 \rangle} \nonumber \\
&+ \underbrace{ \left[ (a_{pq,kl}^{tu} + a_{kl,pq}^{tu}) (a_{ij,rs}^{vw} + a_{rs,ij}^{vw}) - (a_{pq,rs}^{tu} + a_{rs,pq}^{tu}) (a_{ij,kl}^{vw} + a_{kl,ij}^{vw}) \right] g^{pq,rs} g_{tu,vw} }_{\langle 3 \rangle}, \nonumber
\end{align}
where $1 \leq p < q \leq n, 1 \leq r < s \leq n, 1 \leq t < u \leq n, 1 \leq v < w \leq n$.

With the help of a computer algebra system, we can substitute (\ref{str-consts}) and a double-indexed version of (\ref{adj-const}) into this rather unwieldy formula to obtain the following enormous expressions.
%input from ricci_ijkl_unsimplified
\footnotesize
\begin{align*}
4 R_{ij,kl}\langle 1 \rangle 
&= \sum_{\substack{ 1 \leq p < q \leq n \\ 1 \leq r < s \leq n \\ 1 \leq t < u \leq n \\ 1 \leq v < w \leq n}} 
\left\{ 
%\left\{ 
%\left.
\begin{array}{l}
-   g_{tu,vw} g^{pq,rs} \delta_{il} \delta_{jw} \delta_{kv} \delta_{ps} \delta_{qu} \delta_{rt} \\
+   g_{tu,vw} g^{pq,rs} \delta_{iv} \delta_{jk} \delta_{lw} \delta_{ps} \delta_{qu} \delta_{rt} \\
-   g_{tu,vw} g^{pq,rs} \delta_{is} \delta_{jw} \delta_{kq} \delta_{lu} \delta_{pt} \delta_{rv} \\
- 2 g_{tu,vw} g^{pq,rs} \delta_{iq} \delta_{ju} \delta_{ks} \delta_{lw} \delta_{pt} \delta_{rv} \\
+   g_{tu,vw} g^{pq,rs} \delta_{is} \delta_{jw} \delta_{kt} \delta_{lp} \delta_{qu} \delta_{rv} \\
+ 2 g_{tu,vw} g^{pq,rs} \delta_{it} \delta_{jp} \delta_{ks} \delta_{lw} \delta_{qu} \delta_{rv} \\
%\end{array} %\right. 
%\left\{ 
%\begin{array}{l}
+   g_{tu,vw} g^{pq,rs} \delta_{il} \delta_{jw} \delta_{kv} \delta_{pt} \delta_{qr} \delta_{su} \\
-   g_{tu,vw} g^{pq,rs} \delta_{iv} \delta_{jk} \delta_{lw} \delta_{pt} \delta_{qr} \delta_{su} \\
+ 2 g_{tu,vw} g^{pq,rs} \delta_{iq} \delta_{ju} \delta_{kv} \delta_{lr} \delta_{pt} \delta_{sw} \\
+   g_{tu,vw} g^{pq,rs} \delta_{iv} \delta_{jr} \delta_{kq} \delta_{lu} \delta_{pt} \delta_{sw} \\
-   g_{tu,vw} g^{pq,rs} \delta_{iv} \delta_{jr} \delta_{kt} \delta_{lp} \delta_{qu} \delta_{sw} \\
- 2 g_{tu,vw} g^{pq,rs} \delta_{it} \delta_{jp} \delta_{kv} \delta_{lr} \delta_{qu} \delta_{sw}
\end{array} %\right\}
\right. \\
\intertext{}
4 R_{ij,kl}\langle 2 \rangle 
&= \sum_{\substack{ 1 \leq p < q \leq n \\ 1 \leq r < s \leq n \\ 1 \leq t < u \leq n \\ 1 \leq v < w \leq n}} 
\left\{ 
\begin{array}{l}
- g_{vw,rs} g^{pq,rs} \delta_{iq} \delta_{ju} \delta_{ku} \delta_{lw} \delta_{pt} \delta_{tv} \\
+ g_{vw,rs} g^{pq,rs} \delta_{it} \delta_{jp} \delta_{ku} \delta_{lw} \delta_{qu} \delta_{tv} \\
- g_{vw,pq} g^{pq,rs} \delta_{iu} \delta_{jw} \delta_{ks} \delta_{lu} \delta_{rt} \delta_{tv} \\
+ g_{vw,ij} g^{pq,rs} \delta_{ks} \delta_{lu} \delta_{pu} \delta_{qw} \delta_{rt} \delta_{tv} \\
+ g_{vw,pq} g^{pq,rs} \delta_{iu} \delta_{jw} \delta_{kt} \delta_{lr} \delta_{su} \delta_{tv} \\
- g_{vw,ij} g^{pq,rs} \delta_{kt} \delta_{lr} \delta_{pu} \delta_{qw} \delta_{su} \delta_{tv} \\
+ g_{vw,kl} g^{pq,rs} \delta_{iq} \delta_{ju} \delta_{pt} \delta_{ru} \delta_{sw} \delta_{tv} \\
- g_{vw,kl} g^{pq,rs} \delta_{it} \delta_{jp} \delta_{qu} \delta_{ru} \delta_{sw} \delta_{tv} \\
+ g_{vw,rs} g^{pq,rs} \delta_{iq} \delta_{ju} \delta_{kv} \delta_{lt} \delta_{pt} \delta_{uw} \\
- g_{vw,rs} g^{pq,rs} \delta_{it} \delta_{jp} \delta_{kv} \delta_{lt} \delta_{qu} \delta_{uw} \\
+ g_{vw,pq} g^{pq,rs} \delta_{iv} \delta_{jt} \delta_{ks} \delta_{lu} \delta_{rt} \delta_{uw} \\
- g_{vw,ij} g^{pq,rs} \delta_{ks} \delta_{lu} \delta_{pv} \delta_{qt} \delta_{rt} \delta_{uw} \\
- g_{vw,kl} g^{pq,rs} \delta_{iq} \delta_{ju} \delta_{pt} \delta_{rv} \delta_{st} \delta_{uw} \\
+ g_{vw,kl} g^{pq,rs} \delta_{it} \delta_{jp} \delta_{qu} \delta_{rv} \delta_{st} \delta_{uw} \\
- g_{vw,pq} g^{pq,rs} \delta_{iv} \delta_{jt} \delta_{kt} \delta_{lr} \delta_{su} \delta_{uw} \\
+ g_{vw,ij} g^{pq,rs} \delta_{kt} \delta_{lr} \delta_{pv} \delta_{qt} \delta_{su} \delta_{uw} 
\end{array} 
\right. \\
\intertext{}
4 R_{ij,kl}\langle 3 \rangle 
&= \sum_{\substack{ 1 \leq a < b \leq n \\ 1 \leq c < d \leq n \\1 \leq p < q \leq n \\ 1 \leq r < s \leq n \\ 1 \leq t < u \leq n \\ 1 \leq v < w \leq n}} 
\left\{ \begin{array}{l}
- g_{mn,ab} g_{pq,ef} g_{rs,tu} g^{mn,pq} g^{rs,vw} g^{tu,cd} \delta_{av} \delta_{bl} \delta_{cj} \delta_{df} \delta_{ei} \delta_{kw} \\
- g_{ij,ef} g_{mn,ab} g_{rs,tu} g^{mn,pq} g^{rs,vw} g^{tu,cd} \delta_{av} \delta_{bl} \delta_{cq} \delta_{df} \delta_{ep} \delta_{kw} \\
+ g_{mn,ab} g_{pq,ef} g_{rs,tu} g^{mn,pq} g^{rs,vw} g^{tu,cd} \delta_{av} \delta_{bl} \delta_{ce} \delta_{di} \delta_{fj} \delta_{kw} \\ 
+ g_{ij,ef} g_{mn,ab} g_{rs,tu} g^{mn,pq} g^{rs,vw} g^{tu,cd} \delta_{av} \delta_{bl} \delta_{ce} \delta_{dp} \delta_{fq} \delta_{kw} \\ 
+ g_{mn,ab} g_{pq,ef} g_{rs,tu} g^{mn,pq} g^{rs,vw} g^{tu,cd} \delta_{ak} \delta_{bw} \delta_{cj} \delta_{df} \delta_{ei} \delta_{lv} \\ 
+ g_{ij,ef} g_{mn,ab} g_{rs,tu} g^{mn,pq} g^{rs,vw} g^{tu,cd} \delta_{ak} \delta_{bw} \delta_{cq} \delta_{df} \delta_{ep} \delta_{lv} \\ 
- g_{mn,ab} g_{pq,ef} g_{rs,tu} g^{mn,pq} g^{rs,vw} g^{tu,cd} \delta_{ak} \delta_{bw} \delta_{ce} \delta_{di} \delta_{fj} \delta_{lv} \\ 
- g_{ij,ef} g_{mn,ab} g_{rs,tu} g^{mn,pq} g^{rs,vw} g^{tu,cd} \delta_{ak} \delta_{bw} \delta_{ce} \delta_{dp} \delta_{fq} \delta_{lv} \\ 
+ g_{kl,ef} g_{pq,ab} g_{rs,tu} g^{mn,pq} g^{rs,vw} g^{tu,cd} \delta_{av} \delta_{bn} \delta_{cj} \delta_{df} \delta_{ei} \delta_{mw} \\ 
- g_{kl,ab} g_{pq,ef} g_{rs,tu} g^{mn,pq} g^{rs,vw} g^{tu,cd} \delta_{av} \delta_{bn} \delta_{cj} \delta_{df} \delta_{ei} \delta_{mw} \\ 
+ g_{ij,ef} g_{pq,ab} g_{rs,tu} g^{mn,pq} g^{rs,vw} g^{tu,cd} \delta_{av} \delta_{bn} \delta_{cl} \delta_{df} \delta_{ek} \delta_{mw} \\ 
- g_{ij,ef} g_{kl,ab} g_{rs,tu} g^{mn,pq} g^{rs,vw} g^{tu,cd} \delta_{av} \delta_{bn} \delta_{cq} \delta_{df} \delta_{ep} \delta_{mw} \\ 
- g_{kl,ef} g_{pq,ab} g_{rs,tu} g^{mn,pq} g^{rs,vw} g^{tu,cd} \delta_{av} \delta_{bn} \delta_{ce} \delta_{di} \delta_{fj} \delta_{mw} \\ 
+ g_{kl,ab} g_{pq,ef} g_{rs,tu} g^{mn,pq} g^{rs,vw} g^{tu,cd} \delta_{av} \delta_{bn} \delta_{ce} \delta_{di} \delta_{fj} \delta_{mw} \\ 
- g_{ij,ef} g_{pq,ab} g_{rs,tu} g^{mn,pq} g^{rs,vw} g^{tu,cd} \delta_{av} \delta_{bn} \delta_{ce} \delta_{dk} \delta_{fl} \delta_{mw} \\ 
+ g_{ij,ef} g_{kl,ab} g_{rs,tu} g^{mn,pq} g^{rs,vw} g^{tu,cd} \delta_{av} \delta_{bn} \delta_{ce} \delta_{dp} \delta_{fq} \delta_{mw} \\
- g_{kl,ef} g_{pq,ab} g_{rs,tu} g^{mn,pq} g^{rs,vw} g^{tu,cd} \delta_{am} \delta_{bw} \delta_{cj} \delta_{df} \delta_{ei} \delta_{nv} \\ 
+ g_{kl,ab} g_{pq,ef} g_{rs,tu} g^{mn,pq} g^{rs,vw} g^{tu,cd} \delta_{am} \delta_{bw} \delta_{cj} \delta_{df} \delta_{ei} \delta_{nv} \\ 
- g_{ij,ef} g_{pq,ab} g_{rs,tu} g^{mn,pq} g^{rs,vw} g^{tu,cd} \delta_{am} \delta_{bw} \delta_{cl} \delta_{df} \delta_{ek} \delta_{nv} \\ 
+ g_{ij,ef} g_{kl,ab} g_{rs,tu} g^{mn,pq} g^{rs,vw} g^{tu,cd} \delta_{am} \delta_{bw} \delta_{cq} \delta_{df} \delta_{ep} \delta_{nv} \\ 
+ g_{kl,ef} g_{pq,ab} g_{rs,tu} g^{mn,pq} g^{rs,vw} g^{tu,cd} \delta_{am} \delta_{bw} \delta_{ce} \delta_{di} \delta_{fj} \delta_{nv} \\ 
- g_{kl,ab} g_{pq,ef} g_{rs,tu} g^{mn,pq} g^{rs,vw} g^{tu,cd} \delta_{am} \delta_{bw} \delta_{ce} \delta_{di} \delta_{fj} \delta_{nv} \\ 
+ g_{ij,ef} g_{pq,ab} g_{rs,tu} g^{mn,pq} g^{rs,vw} g^{tu,cd} \delta_{am} \delta_{bw} \delta_{ce} \delta_{dk} \delta_{fl} \delta_{nv} \\ 
- g_{ij,ef} g_{kl,ab} g_{rs,tu} g^{mn,pq} g^{rs,vw} g^{tu,cd} \delta_{am} \delta_{bw} \delta_{ce} \delta_{dp} \delta_{fq} \delta_{nv} \\ 
+ g_{kl,ef} g_{mn,ab} g_{rs,tu} g^{mn,pq} g^{rs,vw} g^{tu,cd} \delta_{av} \delta_{bq} \delta_{cj} \delta_{df} \delta_{ei} \delta_{pw} \\ 
+ g_{ij,ef} g_{mn,ab} g_{rs,tu} g^{mn,pq} g^{rs,vw} g^{tu,cd} \delta_{av} \delta_{bq} \delta_{cl} \delta_{df} \delta_{ek} \delta_{pw} \\ 
- g_{kl,ef} g_{mn,ab} g_{rs,tu} g^{mn,pq} g^{rs,vw} g^{tu,cd} \delta_{av} \delta_{bq} \delta_{ce} \delta_{di} \delta_{fj} \delta_{pw} \\ 
- g_{ij,ef} g_{mn,ab} g_{rs,tu} g^{mn,pq} g^{rs,vw} g^{tu,cd} \delta_{av} \delta_{bq} \delta_{ce} \delta_{dk} \delta_{fl} \delta_{pw} \\ 
- g_{kl,ef} g_{mn,ab} g_{rs,tu} g^{mn,pq} g^{rs,vw} g^{tu,cd} \delta_{ap} \delta_{bw} \delta_{cj} \delta_{df} \delta_{ei} \delta_{qv} \\ 
- g_{ij,ef} g_{mn,ab} g_{rs,tu} g^{mn,pq} g^{rs,vw} g^{tu,cd} \delta_{ap} \delta_{bw} \delta_{cl} \delta_{df} \delta_{ek} \delta_{qv} \\ 
+ g_{kl,ef} g_{mn,ab} g_{rs,tu} g^{mn,pq} g^{rs,vw} g^{tu,cd} \delta_{ap} \delta_{bw} \delta_{ce} \delta_{di} \delta_{fj} \delta_{qv} \\ 
+ g_{ij,ef} g_{mn,ab} g_{rs,tu} g^{mn,pq} g^{rs,vw} g^{tu,cd} \delta_{ap} \delta_{bw} \delta_{ce} \delta_{dk} \delta_{fl} \delta_{qv}
\end{array} \right. 
\end{align*}
\normalsize
Let us describe how obtain something usable from this.  First, the expressions simplify somewhat, due to the presence of myriad Kronecker deltas.  For example,
\[ \sum_{\substack{ 1 \leq p < q \leq n \\ 1 \leq r < s \leq n \\ 1 \leq t < u \leq n \\ 1 \leq v < w \leq n}} 
g_{tu,vw} g^{pq,rs} \delta_{il} \delta_{jw} \delta_{kv} \delta_{ps} \delta_{qu} \delta_{rt} 
= \delta_{il} \sum_{\substack{ 1 \leq p < q \leq n \\ 1 \leq r < p \leq n}} 
g_{rq,kj} g^{pq,rp}. \]
Even after simplifying each term in this way, the result is still hopelessly complicated.  So, in order to analyze it effectively we must assume that the inital metric $g$ is diagonal, and then show that the Ricci tensor stays diagonal.  Then the above expressions can be simplified once again by dropping any terms that vanish due to off-diagonal metric factors.  For example,
\[ \delta_{il} \sum_{\substack{ 1 \leq p < q \leq n \\ 1 \leq r < p \leq n}} 
g_{rq,kj} g^{pq,rp}
= \delta_{il} g_{kj,kj} g^{kj,kk}. \]
Again, we do this for each term, and we must show that the off-diagonal terms of the Ricci tensor vanish.  The indices for such terms satisfy $i \neq k$ or $j \ne l$.  There are several index cases that force terms to vanish:
\begin{enumerate}
\item impossible indexing situations, e.g. $g_{ii,kl}$;
\item both $i = k$ and $j = l$ appear in a factor, or both $i = l$ and $j = k$ appear in a factor;
\item all four indices appear in one metric/metric inverse component.
\end{enumerate}
Using this list, one can then see by inspection that $R_{ij,kl} \langle 1 \rangle$, $R_{ij,kl} \langle 2 \rangle$, and $R_{ij,kl} \langle 3 \rangle$ vanish, so the natural basis for $\mfr{ut}_n \R$ (or the frame for $\UT_n \R$ from Lemma \ref{uni-frame}) is stably Ricci-diagonal.  As before, this means $g$ stays diagonal under the flow.  To see what these diagonal terms are, just replace $k$ with $i$ and $l$ with $j$.  Most of the resulting terms contain factors that obviously vanish, according to the list above, and others can be combined.  Once this is done, we drop back to two indices.  That is, write $g_{ij}$ and $R_{ij}$ to mean $g_{ij,ij}$ and  $R_{ij,ij}$, respectively.  This yields 
\[ 4 R_{ij} = 
-2 \sum_{1 \leq p < i} \frac{g_{pj}}{g_{pi}}
+2 g_{ij}^2 \sum_{i < q < j} \frac{1}{g_{iq} g_{qj}}
-2 \sum_{j < r \leq n} \frac{g_{ir}}{g_{jr}}.
\]

\subsection{Ricci flow and the nilsoliton}\label{subsec-uni-nil}

With this part of the calculation complete, we see that the Ricci flow on $\UT_n \R$ is the system
\begin{equation}\label{uni-RF}
\frac{d}{dt} g_{ij} = 
  \sum_{1 \leq p < i} \frac{g_{pj}}{g_{pi}}
-  g_{ij}^2 \sum_{i < q < j} \frac{1}{g_{iq} g_{qj}}
+ \sum_{j < r \leq n} \frac{g_{ir}}{g_{jr}},
\end{equation}
for $1 \leq i < j \leq n$.

\begin{exmp}
When $n=3$, $\UT_3 \R$ is the familiar Heisenberg group, $\Nil^3$.  The Ricci flow is then the system:
\[ \frac{d}{dt} g_{12} = \frac{g_{13}}{g_{23}}, \quad
\frac{d}{dt} g_{13} = - \frac{g_{13}^2}{g_{12} g_{23}}, \quad
\frac{d}{dt} g_{23} = \frac{g_{13}}{g_{12}}. \]
If we set $A = g_{12}, B = g_{23}$, and $C = g_{13}$, then this becomes
\[ \frac{d}{dt} A = \frac{C}{B}, \quad
\frac{d}{dt} B = \frac{C}{A}, \quad
\frac{d}{dt} C = -\frac{C^2}{AB}, \]
which agrees with the equations from Example \ref{nil3}.  (Those equations were ordered differently to agree with the pattern in Section \ref{sec-heis}.)
\end{exmp}

The goal is now to construct a nilsoliton on each space $\UT_n \R$.  These exist by Lauret's theorem, \ref{lauret-thm} above.  The Iwasawa decomposition of the general linear group is $\Gl_n \R = KAN$, where $K = \Or_n \R$, $A$ is the abelian subgroup of diagonal matrices, and $N = \UT_n \R$.  The quotient $G/K$ is an irreducible symmetric space of non-compact type, and such spaces are all Einstein.  But $G/K \cong AN$, whose Lie algebra is a metric solvable extension of $\mfr{ut}_n \R$.  Thus, Lauret's theorem applies.

Now, due to the complexity of the system (\ref{uni-RF}), we are unable to determine the asymptotics of an arbitrary diagonal solution.  Thus, we cannot use the blowdown method of Section \ref{sec-blowdown}.  Instead, we must make a suitable ansatz.

If we picture a diagonal metric as an upper triangular matrix with zeros on the diagonal (which is natural, given the indices), and extrapolate from low-dimensional cases, we might suspect that metric components along diagonals of the matrix have ``the same'' behavior, and that this behavior (with respect to time) changes in fixed increments from diagonal to diagonal.  The components $g_{ij}$ along any diagonal have the property that the quantity $j-i$ is constant.  This means there should be $n-1$ ``different'' types of behavior.  

We make the ansatz that the components of the solution corresponding to the soliton are of the form
\[ g_{ij}(t) = a_{j-i} t^{1-2(j-i)/n}, \]
for some constants $a_{j-i}$ to be determined shortly.  Then the right side of (\ref{uni-RF}) becomes
\begin{align*}
& \sum_{1 \leq p < i} \frac{g_{pj}}{g_{pi}}
- \sum_{i < q < j} \frac{g_{ij}^2}{g_{iq} g_{qj}}
+ \sum_{j < r \leq n} \frac{g_{ir}}{g_{jr}} \\
&=
  \sum_{1 \leq p < i} \frac{a_{j-p} t^{1-2(j-p)/n}}{a_{i-p} t^{1-2(i-p)/n}}
- \sum_{i < q < j}    \frac{a_{j-i}^2 t^{2-4(j-i)/n}}{a_{q-i} t^{1-2(q-i)/n} a_{j-q} t^{1-2(j-q)/n}} \\
&\qquad \qquad 
+ \sum_{j < r \leq n} \frac{a_{r-i} t^{1-2(r-i)/n}}{a_{r-j} t^{1-2(r-j)/n}} \\
&=
  \sum_{1 \leq p < i} \frac{a_{j-p}}           {a_{i-p}} t^{-2(i-p)/n}
- \sum_{i < q < j}    \frac{a_{j-i}^2}{a_{q-i} a_{j-q}} t^{-2(j-q)/n}
+ \sum_{j < r \leq n} \frac{a_{r-i}}           {a_{r-j}} t^{-2(r-j)/n} \\
&= 
t^{-2(i-p)/2} \left(
  \sum_{1 \leq p < i} \frac{a_{j-p}}{a_{i-p}} 
- \sum_{i <  q < j} \frac{a_{j-i}^2}{a_{q-i} a_{j-q}}
+ \sum_{j < r \leq n} \frac{a_{r-i}}{a_{r-j}} \right).
\end{align*}
The left side is
\[ \frac{d}{dt} a_{j-i} t^{1-2(j-i)/2} = a_{j-i} \left( 1-\frac{2(j-i)}{n} \right) t^{-2(j-i)/2}. \]
The powers of $t$ cancel, and so we must find $a_{j-i}$ such that
\[ a_{j-i} \left( 1-\frac{2(j-i)}{n} \right) =  \sum_{1 \leq p < i} \frac{a_{j-p}}{a_{i-p}} 
- \sum_{i <  q < j} \frac{a_{j-i}^2}{a_{q-i} a_{j-q}}
+ \sum_{j < r \leq n} \frac{a_{r-i}}{a_{r-j}}. \]

For some $A > 0$, set 
\[ a_{j-i} = \frac{A^{j-i}}{n^{j-i-1}}. \]
Then the right side becomes
\begin{align*} 
& \sum_{1 \leq p < i} \frac{a_{j-p}}{a_{i-p}} 
- \sum_{i <  q < j} \frac{a_{j-i}^2}{a_{q-i} a_{j-q}}
+ \sum_{j < r \leq n} \frac{a_{r-i}}{a_{r-j}} \\
&=
  \sum_{1 \leq p < i} \frac{A^{j-p}}{n^{j-p-1}} \frac{n^{i-p-1}}{A^{i-p}} 
- \sum_{i <  q < j} \frac{A^{2j-2i}}{n^{2j-2i-2}} \frac{A^{q-i} A^{j-q}}{n^{q-i-1} n^{j-q-1}}
+ \sum_{j < r \leq n} \frac{A^{r-i}}{n^{r-i-1}} \frac{n^{r-j-1}}{A^{r-j}} \\
&=
  \sum_{1 \leq p < i} \frac{A^{j-i}}{n^{j-i}}
- \sum_{i <  q < j} \frac{A^{j-i}}{n^{j-i}}
+ \sum_{j < r \leq n} \frac{A^{j-i}}{n^{j-i}} \\
&= \left( \frac{A}{n} \right)^{j-i} \left(
  \sum_{1 \leq p < i} 1
- \sum_{i <  q < j} 1
+ \sum_{j < r \leq n} 1 \right) \\
&= \left( \frac{A}{n} \right)^{j-i} \big( n - 2(j-i) \big).
\end{align*}
The left side is
\[ a_{j-i} \left( 1-\frac{2(j-i)}{n} \right) = \frac{A^{j-i}}{n^{j-i-1}} \left( 1-\frac{2(j-i)}{n} \right) 
= \left( \frac{A}{n} \right)^{j-i} \left( n - 2(j-i) \right), \]
as desired.

This means
\[ g(t) = \frac{A^{j-i}}{n^{j-i-1}} t^{1-2(j-i)/n} \, \theta^{ij} \otimes \theta^{ij} \]
is a Ricci flow solution on $\UT_n \R$.  To see that $g(1)$ is a soliton, we need to find diffeomorphisms $\eta_t$ of $\UT_n \R$ such that
\[ g(t) = t \eta_t^* g(1) \]
is also a Ricci flow solution.  In something of a \textit{deus ex machina}, we claim that these diffeomorphisms are of the form
\[ (\eta_t(x))^{ij} = t^{-(j-i)/n} x^{ij}, \]
for $x \in \UT_n \R$.  Considering the coframe from Lemma \ref{uni-frame}, we see that
\begin{align*}
\eta_t^* \Theta_p^{ij}
&= \sum_{k=0}^p \sum_{i < r_1 < \cdots < r_k < p} (x^{i r_1} \circ \eta_t) (x^{r_1 r_2} \circ \eta_t) \cdots (x^{r_k p} \circ \eta_t) \\
&= \sum_{k=0}^p \sum_{i < r_1 < \cdots < r_k < p} t^{-(r_1-i)/n} x^{i r_1} t^{-(r_2-r_1)/n} x^{r_1 r_2} \cdots t^{-(p-r_k)/n} x^{r_k p} \\
&= \sum_{k=0}^p \sum_{i < r_1 < \cdots < r_k < p} t^{-(p-r_k+\cdots-r_1+r_1-i)/n} x^{i r_1} x^{r_1 r_2} \cdots x^{r_k p} \\
&= t^{-(p-i)/n} \Theta_p^{ij}, 
\end{align*}
and so
\begin{align*}
\eta_t^* \theta^{ij} 
&= d(x^{ij} \circ \eta_t) - \sum_{i<p<j} \left( 2(x^{ip} \circ \eta_t) - (\Theta^{ij}_p \circ \eta_t) \right) d(x^{pj} \circ \eta_t) \\
&= t^{-(j-i)/n} dx^{ij} - \sum_{i<p<j} \left( 2 t^{-(p-i)/n} x^{pi} - t^{-(p-i)/n} \Theta^{ij}_p \right) t^{-(j-p)/n} dx^{pj} \\
&= t^{-(j-i)/n} \theta^{ij}.
\end{align*}
Now we have
\begin{align*}
t \eta_t^* g(1)
&= t \eta_t^*(g(1)) \, \eta_t^* \theta^{ij} \otimes \eta_t^* \theta^{ij} \\
&= \frac{A^{j-i}}{n^{j-i-1}} t^{1-2(j-i)/n} \, \theta^{ij} \otimes \theta^{ij} \\
&= g(t)
\end{align*}
as required.  Thus, $g(t)$ is an expanding Ricci soliton with respect to the diffeomorphisms $\{\eta_t\}$ just described.

Set $A = 1$.  To conclude, we have another theorem.  

\begin{thm*}[\ref{thm:uni}]
Let $\UT_n \R$ be the Lie group of real $n \times n$ unitriangular matrices, with coordinates as in (\ref{uni-coords}) and coframe $\{\theta^{ij}\}$ as in Lemma \ref{uni-frame}.  Then the family of metrics $g(t) = g_{ij,ij}(t) \, \theta^{ij} \otimes \theta^{ij}$, where 
\[ g_{ij,ij}(t) = \frac{1}{n^{j-i-1}} t^{1-2(j-i)/n}, \]
is a Ricci flow solution on $\UT_n \R$.  The metric $g(1)$ is a nilsoliton with respect to the diffeomorphisms $\eta_t$, where
\[ (\eta_t(x))^{ij} = t^{-(j-i)/n} x^{ij}. \]
\end{thm*}

\begin{rem}
It would appear that we have constructed a family of soliton metrics $\{g_A\}$ depending on the parameter $A$, but it is easy to see that there is a Lie algebra automorphism\footnote{An automorphism $\Phi$ acts on a left-invariant metric $g$ by $\Phi \cdot g = g(\Phi^{-1} \cdot, \Phi^{-1} \cdot)$.} $\Phi_A$, such that $g_1(t) = \Phi_A \cdot g_A(1)$, which means they are equivalent as required by Theorem 3.5 in \cite{Lauret2001}.  %Namely, represent the map as a diagonal matrix $\Phi_A$, with $(\Phi_A)_{ij,ij} = \sqrt{A^{j-i}}$.
\end{rem}

\appendix

\section{Curvature of Lie groups}\label{app-curv}

In this section, we recall some general facts about the geometry of Lie groups with left-invariant metrics, and derive the formula for the Ricci tensor that was used above.  

Let $\langle \cdot, \cdot \rangle$ be a left-invariant metric on a Lie group $G$, which is equivalent to an inner product on the Lie algebra $\mfr{g} = \Lie(G)$.  Let $\nabla$ denote the Levi-Civita connection for the metric, and Let $X,Y,Z,W \in \mfr{g}$.  Recall that $\ad_X = [X,\cdot]$, and its adjoint with respect to $\langle \cdot, \cdot \rangle$ is defined by
\[ \langle (\ad_X)^* Y, Z \rangle = \langle Y, \ad_X Z \rangle. \]

\begin{rem}
Formulas like those in the following propositions appear throughout the literature (e.g., \cite{Besse2008} and \cite{CheegerEbin2008}).  Most of these, however, are derived with the goal of expressing the various related curvatures with respect to a fixed orthonormal basis.  As we are working with evolving metrics, with no initial assumptions on orthonormality, it is more convenient to have curvature formulas that do not depend on an orthonormal basis.
\end{rem}

\begin{prop}\label{nab-rm-prop} We have the following formulas for $\nabla$ and the Riemannian curvature tensor:
\begin{enumerate}
\item[(a)] ${\displaystyle \nabla_X Y = \frac{1}{2} \big( \ad_X Y - (\ad_X)^* Y - (\ad_Y)^* X \big)},$
\smallskip
\item[(b)] $\langle R(X,Y)Z, W \rangle = \langle \nabla_X Z, \nabla_Y W \rangle - \langle \nabla_Y Z, \nabla_X W \rangle - \langle \nabla_{[X,Y]} Z, W \rangle.$
\end{enumerate}
\end{prop}

This result is standard, so we omit the proof.  Now, the maps $\mapelts{(X,Y)}{\ad_X Y}$ and $\mapelts{(X,Y)}{(\ad_X)^* Y}$ are bilinear maps $\map{\mfr{g} \times \mfr{g}}{\mfr{g}}$.  Define 
\fndef{U}{\mfr{g} \times \mfr{g}}{\mfr{g}}{(X,Y)}{-\frac{1}{2} \Big( (\ad_X)^* Y + (\ad_Y)^* X \Big)}
This is symmetric, bilinear, and $U(X,X) = -(\ad_X)^* X$.  It is useful in computing the Riemannian curvature tensor, as we shall see.

\begin{prop}\label{rm40-prop} The Riemannian curvature $(4,0)$-tensor on $G$ is given by
\begin{align*}
&4 \langle R(X,Y)Z,W \rangle \\
&= 2 \langle [X,Y],[Z,W] \rangle + \langle [X,Z],[Y,W] \rangle - \langle [X,W],[Y,Z] \rangle \\
&\quad - \langle [[X,Y],Z],W \rangle + \langle [[X,Y],W],Z \rangle - \langle [[Z,W],X],Y \rangle + \langle [[Z,W],Y],X \rangle \\
&\quad + 4 \langle U(X,Z),U(Y,W) \rangle - 4 \langle U(X,W),U(Y,Z) \rangle.
\end{align*}
As a special case, 
\begin{align*}
\langle R(X,Y)Y,X \rangle 
&= \frac{1}{4} \| (\ad_X)^* Y + (\ad_Y)^* X \|^2 - \langle (\ad_X)^* X, (\ad_Y)^* Y \rangle \\
&\quad - \frac{3}{4} \| [X,Y] \|^2 - \frac{1}{2} \left\langle [[X,Y],Y],X \right\rangle - \frac{1}{2} \left\langle [[Y,X],X],Y \right\rangle,
\end{align*}
which is the sectional curvature $K(X \wedge Y)$ if $X$ and $Y$ are orthonormal.
\end{prop}

The derivation of the first formula is straight-forward, relying mainly on Proposition \ref{nab-rm-prop} and various Lie bracket manipulations.  The second formula follows immediately from the first.

Let $\{e_i\}$ be a basis for $\mfr{g}$.  Then we write
\[ \ad_{e_i} e_j = c_{ij}^k e_k, \quad (\ad_{e_i})^* e_j = a_{ij}^k e_k, \quad \langle e_i, e_j \rangle = g_{ij}. \]
We can use this to write the above formulas in terms of components.

\begin{cor}
\begin{enumerate}
\item[(a)] If $\nabla_{e_i} e_j = \gamma_{ij}^k$, then
\[ \gamma_{ij}^k = \frac{1}{2} g^{kl} \big( c_{ij}^m g_{lm} - c_{il}^m g_{jm} - c_{jl}^m g_{im} \big);  \]
\item[(b)] The components of the Riemann curvature $(4,0)$-tensor satisfy
\begin{align}
4 R_{ijkl}
&= 2 c_{ij}^p c_{kl}^q g_{pq} + c_{ik}^p c_{jl}^q g_{pq} - c_{il}^p c_{jk}^q g_{pq} \nonumber \\
&\quad - c_{ij}^p c_{pk}^q g_{ql} + c_{ij}^p c_{pl}^q g_{qk} - c_{kl}^p c_{pi}^q g_{qj} + c_{kl}^p c_{pj}^q g_{qi} \nonumber  \\
&\quad + (a_{ik}^p + a_{ki}^p) (a_{jl}^q + a_{lj}^q) g_{pq} - (a_{il}^p + a_{li}^p) (a_{jk}^q + a_{kj}^q) g_{pq}. \label{rmcomp}
\end{align}
\item[(c)] The components of the Ricci curvature $(2,0)$-tensor satisfy
\begin{align}
4 R_{ij} 
&= \big( 2 c_{ki}^p c_{jm}^q g_{pq} + c_{kj}^p c_{im}^q g_{pq} - c_{km}^p c_{ij}^q g_{pq} \nonumber \\
&\quad - c_{ki}^p c_{pj}^q g_{qm} + c_{ki}^p c_{pm}^q g_{qj} - c_{jm}^p c_{pk}^q g_{qi} + c_{jm}^p c_{pi}^q g_{qk} \nonumber \\
&\quad + (a_{kj}^p + a_{jk}^p) (a_{im}^q + a_{mi}^q) g_{pq} - (a_{km}^p + a_{mk}^p) (a_{ij}^q + a_{ji}^q) g_{pq} \big) g^{km} . \label{riccomp}
\end{align}
\item[(d)] The sectional curvature $K(e_i \wedge e_j)$ satisfies
\begin{align*}
4 K_{ij} 
&= \Big(3 c_{ij}^p c_{ji}^q g_{pq} - c_{ij}^p c_{pj}^q g_{qi} + c_{ij}^p c_{pi}^q g_{qj} - c_{ji}^p c_{pi}^q g_{qj} + c_{ji}^p c_{pj}^q g_{qi} \\
&\quad + (a_{ij}^p + a_{ji}^p) (a_{ji}^q + a_{ij}^q) g_{pq} - (a_{ii}^p + a_{ii}^p) (a_{jj}^q + a_{jj}^q) g_{pq}\Big)/(g_{ii} g_{jj}-g_{ij}^2). 
\end{align*}
\item[(e)] The scalar curvature satisfies
\begin{align*}
4 S  
&= \big( 2 c_{ki}^p c_{jm}^q g_{pq} + c_{kj}^p c_{im}^q g_{pq} - c_{km}^p c_{ij}^q g_{pq} \\
&\quad - c_{ki}^p c_{pj}^q g_{qm} + c_{ki}^p c_{pm}^q g_{qj} - c_{jm}^p c_{pk}^q g_{qi} + c_{jm}^p c_{pi}^q g_{qk} \\
&\quad + (a_{kj}^p + a_{jk}^p) (a_{im}^q + a_{mi}^q) g_{pq} - (a_{km}^p + a_{mk}^p) (a_{ij}^q + a_{ji}^q) g_{pq} \big) g^{ij} g^{km}.
\end{align*}
\end{enumerate}
\end{cor}

We finally note that the ``adjoint structure constants'' $a_{ij}^k$ can be expressed in terms of $c_{ij}^k$ and $g_{ij}$, by using the definition of $\ad^*$:% as follows:
%\[ \begin{array}{rrcl}
%                 & \langle (\ad_{e_i})^* e_j, e_l \rangle &= &\langle e_j, \ad_{e_i} e_l \rangle \\
%\Longrightarrow  & \langle a_{ij}^k e_k, e_l \rangle      &= &\langle e_j, c_{il}^k e_k \rangle \\
%\Longrightarrow  & a_{ij}^k g_{kl}                        &= &c_{il}^k g_{jk} \\
%\Longrightarrow  & a_{ij}^k g_{kl} g^{lm}                 &= &c_{il}^k g_{jk} g^{lm} \\
%\Longrightarrow  & a_{ij}^m                               &= &c_{il}^k g_{jk} g^{lm} \\
%\end{array} \]
%and exchanging the roles of $k$ and $m$ gives
\begin{equation}\label{adj-const}
a_{ij}^k = c_{il}^m g_{jm} g^{kl}.
\end{equation}
This formula makes it possible to eliminate the $a_{ij}^k$ from the curvature formulas.%, if convenient.

%%% -------------------------------------------------------------------
%%% bibliography
%%%-------------------------------------------------------------------

\bibliography{../refs}

\end{document}